\documentclass[a4paper,leqno,12pt]{amsart}
\usepackage[top=25mm, bottom=25mm, left=20mm, right=20mm]{geometry}
\usepackage{amsmath,amstext,amssymb,enumitem,mathtools}
\usepackage[utf8]{inputenc}

\usepackage[hidelinks]{hyperref}
\usepackage{nameref}
\hypersetup{colorlinks = true, urlcolor = blue, linkcolor = blue, citecolor = red}

\newcommand{\HH}{\mathbb{H}}
\newcommand{\NN}{\mathbb{N}}
\newcommand{\RR}{\mathbb{R}}

\numberwithin{equation}{section}
\theoremstyle{plain}
\newtheorem{theorem}[equation]{Theorem}
\newtheorem{lemma}[equation]{Lemma}
\newtheorem{definition}[equation]{Definition}

\newcommand{\dint}{\text{\normalfont{d}}} 
\DeclarePairedDelimiter{\norm}{\lVert}{\rVert}

\begin{document}
	
	\title[Sharp estimates for Jacobi heat kernels in conic domains]{Sharp estimates for Jacobi heat kernels in conic domains}
	
	\author{Dawid Hanrahan}
	\address{Dawid Hanrahan (\textnormal{dawid.hanrahan@pwr.edu.pl})
		\newline WUST -- Wroc\l aw University of Science and Technology, 50-370 Wroc\l aw, Poland}		
	\author{Dariusz Kosz}
	\address{Dariusz Kosz (\textnormal{dkosz@bcamath.org.pl})
		\newline BCAM -- Basque Center for Applied Mathematics, 48009 Bilbao, Spain \newline WUST -- Wroc\l aw University of Science and Technology, 50-370 Wroc\l aw, Poland}		
	
	\begin{abstract} We prove genuinely sharp estimates for the Jacobi heat kernels introduced in the context of the multidimensional cone $\mathbb{V}^{d+1}$ and its surface $\mathbb{V}^{d+1}_0$. To do so, we combine the theory of Jacobi polynomials on the cone explored by Xu with the recent techniques by Nowak, Sj\"ogren, and Szarek, developed to find genuinely sharp estimates for the spherical heat kernel.     
	
	
	\end{abstract}
	
\maketitle

\section{Introduction}

Heat kernels are important objects in mathematics and physics. Despite hundreds of articles devoted to studying them, it was only very recently that the development of techniques allowed the so-called genuinely sharp estimates to be given in settings other than a few classical ones such as the hyperbolic space $\HH^{d+1}$, see \cite{DM88}. For the spherical heat kernel genuinely sharp estimates were obtained in \cite{NSS18}, while the Jacobi heat kernels for all compact rank-one symmetric spaces, including the classical domain $[-1,1]$, were investigated in \cite{NSS21}, with the aid of some tools elaborated earlier in \cite{NS13}. 

The aim of this article is to find genuinely sharp estimates for the Jacobi heat kernels on the multidimensional cone and its surface introduced by Xu in \cite{Xu20}, see Theorems~\ref{T1}~and~\ref{T2}. In the proofs we make use of the ideas invented in \cite{NSS21}. 
Let us also explain that ``genuinely sharp'' means that  
the exact expressions which control the heat kernels simultaneously from above and below are given. We emphasize that so far such a high level of precision has only been achieved in very few settings. In addition to the previously mentioned articles, we also refer the reader to \cite{BM16, MSZ16, MS20, Se22}.

Before stating our results, let us describe a general context in which heat kernels arise. In the next paragraph we follow \cite[Section~3]{DX14} and \cite[Section~1]{Xu20}. 

Given an appropriate weight $\varpi$ defined on a domain $\Omega \subset \RR^d$, one can construct an orthogonal basis of polynomials in $\mathcal L^2(\Omega, \varpi)$. For each $n \in \{0,1,\dots\}$, let $\mathcal V_n(\varpi)$ be the subspace of $\mathcal L^2(\Omega, \varpi)$ spanned by all basis vectors that are polynomials of degree $n$. Then the orthogonal projection ${\rm proj}_n \colon \mathcal L^2(\Omega, \varpi) \to \mathcal V_n(\varpi)$ turns out to be of the form
\begin{equation*}
{\rm proj}_n f(x) \coloneqq \int_\Omega f(y) P_n(\varpi; x,y) \, \dint \varpi(y)
\end{equation*}
for some integral kernel $P_n(\varpi; \cdot, \cdot) \colon \Omega^2 \to \RR$ called the reproducing kernel of $\mathcal V_n(\varpi)$. Among many possible choices of $(\Omega,\varpi)$, we focus on these for which the following two properties hold.
\begin{enumerate} \itemsep=0.1cm
\item \label{item1} There exists a linear second order differential operator $\mathcal  D$ such that all basis polynomials are its eigenfunctions with eigenvalues depending only on the degrees $n$.
We call $\mathcal D$ a~diffusion operator. Its domain consists of all $f \in \mathcal L^2(\Omega, \varpi)$ for which $\mathcal D f$, defined formally by using orthogonal expansions, can be identified with elements of $\mathcal L^2(\Omega, \varpi)$.
\item \label{item2} Each $P_n(\varpi, \cdot, \cdot)$ has a ``computable'' closed-form formula.
\end{enumerate}
If \eqref{item1} holds, then for each $\tau \in (0, \infty)$ one can define the associated heat kernel by
\begin{equation*}
h_{\tau}(\varpi;x,y)    
\coloneqq \sum \limits_{n=0}^{\infty} e^{- \tau \lambda_n^{\mathcal D}} \, P_{n}(\varpi;x,y)
\end{equation*}
with $\lambda_n^{\mathcal D}$ being the eigenvalues of $\mathcal D$ corresponding to $\mathcal V_n(\varpi)$. Informally speaking, $h_\tau(\varpi; x,y)$ measures the heat flow between $x$ and $y$ in time $\tau$, when $\mathcal D$ describes heat diffusion. It thus can be used to produce the solution to an initial value problem for the related heat equation. If \eqref{item2} holds, then one can hope to find precise estimates for the size of $h_{\tau}(\varpi;x,y)$ for given $x,y,\tau$. 

In principle, situations in which both \eqref{item1} and \eqref{item2} occur are very rare. In \cite{Xu20} the studied domains were the cone $\mathbb V^{d+1}$ and its surface $\mathbb V^{d+1}_0$ with any given $d \in \{2,3, \dots \}$, see \eqref{def1} and \eqref{def2} for the definitions. In both cases Xu was able to find suitable weights $\varpi$ allowing the two properties to happen, and gave the formulas for the associated heat kernels. The author named the latter objects Jacobi heat kernels because of their clear association with the classical Jacobi setting, which in turn was due to the particular form of $\varpi$.

In the following subsections we recall some parts of the theory developed in \cite{Xu20}. It is worth mentioning that in order to find $\varpi$ several other objects, such as the spherical harmonics or the classical Jacobi polynomials on $[-1,1]$, were used. This resembles the fact that multidimensional cones inherit geometrical properties of both intervals and Euclidean balls. Xu presented a very detailed approach to the subject. In particular, simpler settings -- intervals, triangles, and balls -- were considered first, see \cite[Section~2]{Xu20}, and only then suitable orthogonal polynomials and the associated differential operators on $\mathbb{V}^{d+1}$ and $\mathbb{V}^{d+1}_0$ were defined. We did not want to repeat this content line by line so only the most important formulas, from the standpoint of our results, are collected. For more detailed information or intuitions behind the formulas we refer the reader to \cite{Xu20} or to the books \cite{DX13, DX14}. 
   
\subsection{Jacobi heat kernel on $\mathbb{V}^{d+1}$} This material comes from \cite[Sections~3~and~4]{Xu20}, where one should specify $\beta=0$\footnote{In both settings the additional parameter $\beta$ corresponds to the factor $t^\beta$ in the weight. In this article it is fixed due to the reasons mentioned in \ref{d} in Subsection~\ref{S1.3}.}. 
Given $d \in \{2,3, \dots\}$, consider the domain
\begin{align} \label{def1}
\mathbb{V}^{d+1} \coloneqq \big\{(x, t) \in \mathbb{R}^{d} \times \mathbb{R} :  \norm{x} \le t, \, t \in [0,1] \big\}
\end{align}
contained in $\RR^{d+1}$, equipped with the weight
\begin{equation*}
W_{\mu, \gamma}(x, t) \coloneqq (t^2 - \norm{x}^2)^{\mu - \frac{1}{2}} (1 - t)^{\gamma},
\end{equation*}
where $\mu \in (- \frac{1}{2}, \infty), \gamma \in (-1, \infty)$ are fixed parameters.
Then, for each $n \in \{0,1, \dots\}$, the space  $\mathcal{V}_{n}(W_{\mu, \gamma})$ of orthogonal polynomials of degree $n$ related to $W_{\mu, \gamma}$, is described in terms of the so-called Jacobi polynomials on the cone. Moreover, there exists a suitable operator $\mathcal D_{\mu, \gamma}$ acting on a subspace of $\mathcal L^2(\mathbb{V}^{d+1}, c_{\mu, \gamma} W_{\mu, \gamma})$, where $c_{\mu, \gamma}$ is the normalizing constant\footnote{That is, $c_{\mu, \gamma}$ is the unique constant for which $c_{\mu, \gamma} W_{\mu, \gamma}$ becomes a probability measure on $\mathbb{V}^{d+1}$.}, with all $\mathcal{V}_{n}(W_{\mu, \gamma})$ being its eigenspaces. The associated eigenvalues equal $-n(n+2\mu + \gamma + d)$.

We do not use the formula for $\mathcal D_{\mu, \gamma}$ later on but for the sake of completeness we recall that 
\begin{align*}
\mathcal D_{\mu, \gamma} & \coloneqq  
t (1-t) \partial_t^2 + 2(1-t) \langle x, \nabla_x \rangle \partial_t
+ \sum_{i=1}^d (t - x_i^2) \partial_{x_i}^2
- 2 \sum_{i < j} x_i x_j \partial_{x_i} \partial_{x_j} \\ 
& \quad + (2 \mu + d) \partial_{t} - (2 \mu  + \gamma + d + 1)
(\langle x, \nabla_x \rangle + t \partial_t ),
\end{align*}
where $\nabla_x$ is the gradient in $x$, see \cite[Theorem~3.2]{Xu20}.

If $\mu \geq 0$ and $\gamma \geq -\frac{1}{2}$, then $P_{n} (W_{\mu, \gamma};(x, t), (y, s))$, the reproducing kernel of $\mathcal{V}_{n}(W_{\mu, \gamma})$, is given by the following integral 
(cf.~\cite[(4.9)]{Xu20})    
\begin{align*} \label{RK1}
P_{n}\big(W_{\mu, \gamma};(x, t), (y, s)\big) \coloneqq \int \limits_{[-1, 1]^3} Z_{2n}^{2\alpha + \gamma + 1} \big(\xi(x, t, y, s;u, v_1, v_2)\big)  
\, \dint \Pi_{\mu - \frac{1}{2}}(u) \dint \Pi_{\alpha - \frac{1}{2}}(v_1) \dint \Pi_{\gamma}(v_2).
\end{align*} 
Here $\dint \Pi_{\nu}(w) \coloneqq c_\nu (1-w^2)^{\nu - \frac{1}{2}} \, \dint w$ for $\nu \in (-\frac{1}{2}, \infty)$ and $w \in [-1,1]$, where $c_\nu$ is the normalizing constant, while $\Pi_{-\frac{1}{2}}$ is the mean of Dirac deltas $\frac{1}{2}(\delta_{-1} + \delta_{1})$. Also, $\alpha \coloneqq \mu + \frac{d - 1}{2}$ and
\begin{align*}
\xi(x, t, y, s;u, v_1,v_2) \coloneqq v_1 \sqrt{\tfrac{1}{2}\big(st + \langle x, y \rangle + \sqrt{t^2 - \norm{x}^2} \sqrt{s^2 - \norm{y}^2}u \big)} 
+ v_2 \sqrt{1-t} \sqrt{1-s}.
\end{align*}
We note that $|\xi(x, t, y, s;u, v_1,v_2)| \leq 1$, as shown in the end of the proof of \cite[Theorem~4.3]{Xu20}. Finally, for $n \in \NN_0$, $\lambda \in [0, \infty)$, and $w \in [-1,1]$ we also use the special function\footnote{Here we only need $\lambda \in [1, \infty)$ but the wider range will be used in the context of $\mathbb V_0^{d+1}$.}
\begin{equation} \label{Gegen}
Z_{n}^{\lambda + \frac{1}{2}}(w) \coloneqq \frac{C_{n}^{\lambda + \frac{1}{2}}(1) C_{n}^{\lambda + \frac{1}{2}}(w)}{h^{\lambda + \frac{1}{2}}_n}  = \frac{P_{n}^{\lambda, \lambda}(1)P_{n}^{\lambda, \lambda}(w)}{h_{n}^{\lambda, \lambda}}.
\end{equation}
Here $C_{n}^{\lambda + \frac{1}{2}}$ is the Gegenbauer polynomial of degree $n$, and  
$P_n^{\lambda, \lambda}$ is the classical Jacobi polynomial on $[-1,1]$, while $h^{\lambda + \frac{1}{2}}_n$ and $h_n^{\lambda, \lambda}$ are the squares of their norms in the space $\mathcal L^2 ([-1,1], \dint \Pi_{\lambda + \frac{1}{2}} )$. In \cite[Subsection~2.2]{Xu20} the  function $Z_{n}^{\lambda + \frac{1}{2}}$ is defined through the Gegenbauer polynomial but here we will use the last expression in \eqref{Gegen} instead. The two expressions in \eqref{Gegen} are equal because $P_n^{\lambda, \lambda}$ is a constant multiple of $C_{n}^{\lambda + \frac{1}{2}}$, see \cite[(4.7.1)]{Sz75}. We recall that $P_n^{\lambda, \lambda}$, and hence also $Z_n^{\lambda + \frac{1}{2}}$, is of the same parity as $n$.

\begin{definition}
	\label{HK1}
	Let $\mu \in [0, \infty)$, $\gamma \in [-\frac{1}{2}, \infty)$, and $\alpha \coloneqq \mu + \frac{d - 1}{2}$. Then, for each $\tau \in (0, \infty)$, the associated Jacobi heat kernel on $\mathbb{V}^{d+1}$ is given by
	\begin{equation*}
	h_{\tau} \big(W_{\mu, \gamma};(x, t), (y, s) \big) 
	\coloneqq \sum \limits_{n=0}^{\infty} e^{- \tau n (n + 2\mu + \gamma + d)} \, P_{n}\big(W_{\mu, \gamma};(x, t), (y, s)\big).
	\end{equation*}
\end{definition} 

In \cite[Definition~5.1]{Xu21} the following distance function on $\mathbb{V}^{d+1}$ was introduced
\begin{equation*}
{\rm dist}_{\mathbb{V}^{d+1}} \big((x,t),(y,s) \big) \coloneqq
\arccos \Big( \sqrt{\tfrac{1}{2} \big( st + \langle x, y \rangle + \sqrt{t^2 - \norm{x}^2} \sqrt{s^2 - \norm{y}^2} \big) } + \sqrt{1 - t} \sqrt{1 - s} \Big).
\end{equation*}
According to this, for the Jacobi heat kernel on $\mathbb{V}^{d+1}$ we shall show the following result. 

\begin{theorem} \label{T1} 
	Let $\mu \in [0, \infty), \gamma \in [-\frac{1}{2}, \infty)$, and $\alpha \coloneqq \mu + \frac{d - 1}{2}$. Then $h_{\tau} (W_{\mu, \gamma};(x, t), (y, s) )$, the Jacobi heat kernel on $\mathbb{V}^{d+1}$, is comparable to 
	\begin{align*} 
	&\tau^{-\alpha + \mu - 1} \big( \sqrt{1-t} \sqrt{1-s} \vee \tau \big)^{-\gamma - \frac{1}{2}} \big( \sqrt{st + \langle x, y \rangle} \vee \tau \big)^{-\alpha} \\
	&\times \Big( \sqrt{ (t^2 - \norm{x}^2)(s^2 - \norm{y}^2) (st + \langle x, y \rangle)^{-1}} \vee \tau \Big)^{-\mu}  \exp \big\{ -\tfrac{1}{\tau} {\rm dist}^2_{\mathbb{V}^{d+1}} \big((x,t),(y,s) \big) \big\}
	\end{align*}
	for $\tau \in (0, 4]$, and to $1$ for $\tau \in [4, \infty)$, uniformly in $(x, t), (y, s) \in \mathbb{V}^{d+1}$.
\end{theorem}

\subsection{Jacobi heat kernel on $\mathbb{V}^{d+1}_0$} This material comes from \cite[Sections~7~and~8]{Xu20}, where one should specify $\beta=-1$. 
Given $d \in \{2,3, \dots\}$, consider the domain
\begin{align} \label{def2}
\mathbb{V}^{d+1}_0 \coloneqq \big\{(x, t) \in \mathbb{R}^{d} \times \mathbb{R} :  \norm{x} = t, \, t \in [0,1] \big\}
\end{align} 
contained in $\RR^{d+1}$, equipped
with the weight
\begin{equation*}
\varphi_{\gamma}(t) \coloneqq t^{-1}(1 - t)^{\gamma},
\end{equation*}
associated with the $d$-dimensional Lebesgue measure $\rho(x,t)$ on $\mathbb{V}^{d+1}_0$, where $\gamma \in (-1, \infty)$ is a~fixed parameter.
Then, for each $n \in \{0,1, \dots\}$, the space $\mathcal{V}_{n}(\varphi_{\gamma})$ of orthogonal polynomials of degree $n$ related to $\varphi_{\gamma}$, is described in terms of the so-called Jacobi polynomials on the surface of the cone. Moreover, there exists a suitable operator $\mathcal D_{\gamma}$ acting on a subspace of $\mathcal L^2(\mathbb{V}^{d+1}, c_{\gamma} \varphi_{ \gamma})$, where $c_{\gamma}$ is the normalizing constant, with all $\mathcal{V}_{n}(\varphi_{\gamma})$ being its eigenspaces. The associated eigenvalues equal $-n(n+ \gamma + d -1)$.

As before, for the sake of completeness we recall that 
\begin{align*}
\mathcal D_{\gamma} & \coloneqq  
t (1-t) \partial_t^2 + \big(d-1- (d + \gamma)t \big) \partial_t + t^{-1} \Delta_0^{(x)},
\end{align*}
where $\Delta_0^{(x)}$ is the Laplace--Beltrami operator on the unit sphere $\mathbb S^{d-1}$ taken in $x$. More precisely, given a polynomial $f \in \mathcal{V}_{n}(\varphi_{\gamma})$ we define $\Delta_0^{(x)} f$ for each $t$ separately, referring to the function $\frac{x}{t} \mapsto f(x,t)$ defined on $\mathbb S^{d-1}$ through the projection $(x,t) \mapsto \frac{x}{t}$, and mapping the result back through the inverse transformation $\frac{x}{t} \mapsto (x,t)$, see \cite[Theorem~7.2]{Xu20} for details.

If $\gamma \geq -\frac{1}{2}$, then $P_{n} (\varphi_{\gamma}; (x, t), (y, s))$, the reproducing kernel of $\mathcal{V}_{n}(\varphi_{\gamma})$, is given by the following integral 
(cf.~\cite[(8.5)]{Xu20})    
\begin{align*} \label{RK2} 
P_{n}\big(\varphi_{\gamma}; (x, t), (y, s)\big)
\coloneqq \int \limits_{[-1, 1]^2} Z_{2n}^{\gamma + d-1} \big(\xi(x, t, y, s; v_1,v_2)\big)   
\, \dint \Pi_{\frac{d-3}{2}}(v_1) \dint \Pi_{\gamma}(v_2),
\end{align*}
where $\xi(x, t, y, s; v_1,v_2) \coloneqq  v_1 \sqrt{ \tfrac{1}{2} ( st + \langle x, y \rangle ) } + v_2 \sqrt{1-t} \sqrt{1-s} \in [-1,1]$.

\begin{definition}
	\label{HK2}
	Let $\gamma \in [-\frac{1}{2}, \infty)$. Then, for each $\tau \in (0, \infty)$, the associated Jacobi heat kernel on $\mathbb{V}^{d+1}_0$ is given by
	\begin{equation*}
	h_{\tau} \big(\varphi_{\gamma};(x, t),  (y, s) \big) 
	\coloneqq \sum \limits_{n=0}^{\infty} e^{- \tau n (n + \gamma + d - 1)} \, P_{n}\big(\varphi_{\gamma};(x, t), (y, s)\big).
	\end{equation*}
\end{definition}

In \cite[Definition~4.1]{Xu21} the following distance function on $\mathbb{V}^{d+1}_0$ was introduced
\begin{equation*}
{\rm dist}_{\mathbb{V}_0^{d+1}} \big((x,t),(y,s) \big) \coloneqq
\arccos \Big( \sqrt{\tfrac{1}{2} \big( st + \langle x, y \rangle \big)} + \sqrt{1-t}\sqrt{1-s} \Big).
\end{equation*}
According to this, for the Jacobi heat kernel on $\mathbb{V}_0^{d+1}$ we shall show the following result.

\begin{theorem} \label{T2} 
	Let $\gamma \in [-\frac{1}{2}, \infty)$. Then $h_{\tau} (\varphi_{\gamma}; \, (x,t), (y,s))$, the Jacobi heat kernel on $\mathbb{V}^{d+1}_0$, is comparable to 
	\begin{align*}
	 \tau^{-\frac{d}{2}} \big(\sqrt{1-t}\sqrt{1-s} \vee \tau\big)^{-\gamma - \frac{1}{2}} \big(\sqrt{st + \langle x, y \rangle} \vee \tau\big)^{-\frac{d}{2} + 1}
	 \exp \big\{ -\tfrac{1}{\tau} {\rm dist}^2_{\mathbb{V}_0^{d+1}} \big((x,t),(y,s) \big) \big\}
	\end{align*}
	for $\tau \in (0, 4]$, and to $1$ for $\tau \in [4, \infty)$, uniformly in $(x, t), (y, s) \in \mathbb{V}^{d+1}_0$.
\end{theorem}

\subsection{Comments} \label{S1.3} Several remarks regarding Theorems~\ref{T1}~and~\ref{T2} are in order.
\begin{enumerate}[label=(\alph*)]
	\item The obtained estimates are genuinely sharp.   
	We emphasize that these are ones of the very few instances in which such a high level of precision has been achieved. 
	\item As in more classical situations, the presence of the distance functions in the exponents was expected. The other terms capture the specific behavior of the kernels evaluated at points close to the -- properly understood for $\mathbb V^{d+1}_0$ -- boundaries of the domains.   
	\item Multidimensional cones inherit geometrical properties of both, intervals and Euclidean balls, which is important in the analysis of the associated Jacobi setting. For $d=1$ the geometry simplifies significantly, as the Euclidean component disappears, hence this case is not treated here. For further information see \cite[Section~2.5]{Xu20}.    
	\item \label{d} Wider ranges of parameters in the context of the Jacobi setting can be considered, namely, $\mu \in (-\frac{1}{2}, \infty)$ and $\beta, \gamma \in (-1, \infty)$ with $\alpha \coloneqq \mu+ \frac{\beta + d -1}{2}$ for $\mathbb{V}^{d+1}$ (cf.~\cite[Section~3.1]{Xu20}), or $\beta \in (-d, \infty)$ and $\gamma \in (-1, \infty)$ for $\mathbb{V}^{d+1}_0$ (cf.~\cite[Section~7.1]{Xu20}). However, we were not able to extend the ranges in Theorems~\ref{T1}~and~\ref{T2}. The additional parameter $\beta$, corresponding to the factor $t^\beta$ in the weights, is specified to be either $0$ or $-1$ because only then the subspaces generated by the Jacobi polynomials of degree $n$ become eigenspaces of the related differential operators, see \cite[Remarks~3.1~and~7.1]{Xu20}. The parameters $\mu$ and $\gamma$ are restricted in accordance to the ranges in which useful formulas for the associated reproducing kernels hold, see \cite[Theorems~4.3~and~8.2]{Xu20}.   
	\item We focus only on $\tau \in (0, 4]$ because the uniform estimates for $\tau \in [4, \infty)$ are known. Also, we note that for each $T \in (0, \infty)$ the ranges $(0,T]$ and $[T,\infty)$ can be considered instead, and the same results follow with the implicit constants depending on $T$.
	\item The idea behind Theorems~\ref{T1}~and~\ref{T2} is to express the studied Jacobi heat kernels in terms of the heat kernel related to the classical Jacobi expansions on $[-1,1]$, see Lemmas~\ref{LHK1}~and~\ref{LHK2}. In principle, we follow the strategy proposed in \cite{NSS21} but then it turns out that the $n$-th terms in our original kernels correspond to the $2n$-th terms in the classical heat kernel so that the odd terms, whose integrals equal $0$ thanks to some parity arguments, should be artificially added to the formulas. This phenomenon seems to be new, as it did not appear in the context of spaces considered in \cite{NSS21}.
\end{enumerate}

\subsection*{Acknowledgments} The authors are thankful to Adam Nowak for drawing their attention to the topic discussed in the article, and to Tomasz Z. Szarek for his helpful suggestions.

The authors are also indebted to the anonymous referees for their valuable remarks which led to a significant improvement in the presentation of the results.  

The second author was supported by the Basque Government (BERC 2022-2025), by the Spanish State Research Agency
(CEX2021-001142-S and RYC2021-031981-I), and by the Foundation for Polish Science (START 032.2022).   

\section{Technical preparation} In this section we collect auxiliary results that will help us to prove the main theorems. We first give an alternative formula for the Jacobi heat kernels, following the arguments presented in \cite{NSS21}. The main idea here is to get rid of the oscillatory nature of Jacobi polynomials, and to end up with a certain positive expression which is much easier to deal with.

Given $\lambda \in [0, \infty)$, we recall the formula for the heat kernel $G_{\tau }^{\lambda, \lambda}$ associated with the classical Jacobi polynomials $P_n^{\lambda,\lambda}$ on $[-1,1]$, when the first argument equals $1$. For $\tau \in (0, \infty)$, we have
\begin{equation*}
G_{\tau }^{\lambda, \lambda}(1, w) \coloneqq \sum \limits_{n=0}^{\infty} e^{-\tau  \lambda_{n}} \frac{P_n^{\lambda, \lambda}(1) P_n^{\lambda, \lambda}(w)}{h_n^{\lambda, \lambda}}
= \sum \limits_{n=0}^{\infty} e^{-\tau  \lambda_{n}} Z_n^{\lambda + \frac{1}{2}}(w),
\end{equation*}
where $\lambda_{n} \coloneqq n(n + 2\lambda + 1)$. 

The heat kernels from Definitions~\ref{HK1}~and~\ref{HK2} can be described in terms of $G_{\tau }^{\lambda, \lambda}$. 

\begin{lemma}
	\label{LHK1}
	For each $\tau \in (0,\infty)$, the Jacobi heat kernel from Definition~\ref{HK1} can be given by
	\begin{align*}
	\int \limits_{[-1, 1]^3} G_{\frac{\tau}{4}}^{2\alpha + \gamma + \frac{1}{2},  2\alpha + \gamma + \frac{1}{2}}\big(1, \xi(x, t, y, s; u, v_1, v_2)\big)
	\, \dint \Pi_{\mu - \frac{1}{2}}(u) \dint \Pi_{\alpha - \frac{1}{2}}(v_1) \dint \Pi_{\gamma}(v_2).
	\end{align*}	
\end{lemma}

\begin{proof}
	We expand the left hand side of the formula from Definition~\ref{HK1}. Then we obtain
	\begin{align*}
	\sum \limits_{n=0}^{\infty} e^{- \tau n (n + 2\mu + \gamma + d)}
	\int \limits_{[-1, 1]^3} Z_{2n}^{2\alpha + \gamma + 1}\big(\xi(x, t, y, s; u, v_1, v_2)\big) 
	\, \dint \Pi_{\mu - \frac{1}{2}}(u) \dint \Pi_{\alpha - \frac{1}{2}}(v_1) \dint \Pi_{\gamma}(v_2).
	\end{align*}
	Observe that
	$
	\tau n (n + 2\mu + \gamma + d) = \frac{\tau}{4} 2n(2n + 4\alpha + 2\gamma + 2).
	$
	Using this and the fact that the functions 
	$
	Z_{2n+1}^{2\alpha + \gamma + 1}\big(\xi(x, t, y, s;u, v_1, v_2)\big)
	$
	are odd with respect to $v \coloneqq (v_1,v_2)$, we arrive at
	\begin{align*}
	\sum \limits_{n=0}^{\infty} e^{- \frac{\tau}{4} n(n + 4\alpha + 2\gamma + 2)}
	\int \limits_{[-1, 1]^3} Z_{n}^{2\alpha + \gamma + 1}\big(\xi(x, t, y, s;u, v_1, v_2)\big) 
	\, \dint \Pi_{\mu - \frac{1}{2}}(u) \dint \Pi_{\alpha - \frac{1}{2}}(v_1) \dint \Pi_{\gamma}(v_2).
	\end{align*}
	Then Fubini's theorem gives the claim, since $n(n+4 \alpha + 2 \gamma + 2) = \lambda_{n}$ if $\lambda = 2 \alpha + \gamma + \frac{1}{2}$. 
\end{proof}

\begin{lemma}
	\label{LHK2}
	For each $\tau \in (0,\infty)$, the Jacobi heat kernel from Definition~\ref{HK2} can be given by
	\begin{align*}
	\int \limits_{[-1, 1]^2} G_{\frac{\tau}{4}}^{\gamma+d- \frac{3}{2}, \gamma+d- \frac{3}{2}}\big(1, \xi(x, t, y, s; v_1, v_2) \big) 
	\, \dint \Pi_{\frac{d-3}{2}}(v_1) \dint \Pi_{\gamma}(v_2).
	\end{align*}
\end{lemma}

\begin{proof}
	We expand the left hand side of the formula from Definition~\ref{HK2}. Then we obtain
	\begin{align*}
	\sum \limits_{n=0}^{\infty} e^{- \tau n (n + \gamma + d - 1)}
	\int \limits_{[-1, 1]^2} Z_{2n}^{\gamma + d-1} \big(\xi(x, t, y, s; v_1,v_2)\big) 
	\, \dint \Pi_{\frac{d-3}{2}}(v_1) \dint \Pi_{\gamma}(v_2).
	\end{align*}
	Observe that
	$
	\tau n (n + \gamma + d - 1) = \frac{\tau}{4} 2n(2n + 2\gamma + 2d - 2).
	$
	Using this and the fact that the functions
	$
	Z_{2n+1}^{\gamma + d-1} \big(\xi(x, t, y, s; v_1, v_2)\big)
	$
	are odd with respect to $v \coloneqq (v_1,v_2)$, we arrive at
	\begin{align*}
	\sum \limits_{n=0}^{\infty} e^{- \frac{\tau}{4} n(n + 2\gamma + 2d - 2)}
	\int \limits_{[-1, 1]^2} Z_{n}^{\gamma + d-1} \big(\xi(x, t, y, s; v_1, v_2)\big) 
	\, \dint \Pi_{\frac{d-3}{2}}(v_1) \dint \Pi_{\gamma}(v_2).
	\end{align*}
	Then Fubini's theorem gives the claim, since $n(n + 2 \gamma + 2d - 2) = \lambda_{n}$ if $\lambda = \gamma + d - \frac{3}{2}$. 
\end{proof}

The rest of this section is devoted to recalling various estimates obtained in \cite{NSS18} in the context of the Jacobi heat kernel on $\mathbb{B}^d$. The first one is \cite[Lemma~2.1]{NSS21} with $\xi = 0$.   

\begin{lemma}
	\label{LNSS1}
	Fix $\nu \in [-\frac{1}{2}, \infty)$, $\xi \in \mathbb{R}$, and let $\Phi_{A, B}(w) \coloneqq \arccos(A + Bw)$. Then
	\begin{align*}
	\int \limits_{[0, 1]} \exp \big\{ -\Phi_{A, B}^2(w)  D^{-1} \big\} \, \dint \Pi_{\nu}(w)
	\simeq D^{\nu + \frac{1}{2}} \big(B (\pi -  \Phi_{A, B}(1))^{-1} + D \big)^{- \nu - \frac{1}{2}} \exp \big\{ -\Phi_{A, B}^2(1) D^{-1} \big\}
	\end{align*}
	uniformly in $B \in [0,1]$, $A \in [-1, 1-B]$, and $D \in (0, \infty)$ (with $B (\pi -  \Phi_{A, B}(1))^{-1} = 0$ if $B=0$).
\end{lemma}

\noindent The next estimate follows directly from the proof \cite[Lemma~2.1]{NSS21} with $\xi = 0$.

\begin{lemma}
	\label{LNSS2}
	Fix $\nu \in [-\frac{1}{2}, \infty)$. Then
	\begin{align*}
	\int \limits_{\varphi_1}^{\varphi_0} \exp \big\{ - \tfrac{\psi^2}{D} \big\} ( \cos \varphi_1 - \cos \psi)^{\nu -\frac{1}{2}} \sin \psi \, \dint \psi 
	\simeq D^{\nu+\frac{1}{2}} (\pi - \varphi_1)^{\nu+\frac{1}{2}} \big( \tfrac{(\varphi_0 - \varphi_1) \varphi_0}{(\varphi_0 - \varphi_1) \varphi_0 + D} \big)^{\nu + \frac{1}{2}} \exp \big\{ -\tfrac{\varphi_1^2}{D} \big\}
	\end{align*}
	uniformly in $D \in (0, \infty)$ and $\varphi_0, \varphi_1 \in [0,\pi]$ with $\varphi_1 < \varphi_0$.
\end{lemma}

\begin{proof}
	Indeed, according to the notation from the proof of \cite[Lemma~2.1]{NSS21} with $\xi = 0$ our expression is comparable to
	\[
	D^{\nu + \frac{1}{2}} (\pi - \varphi_1)^{\nu-\frac{1}{2}}  J 
	\simeq D^{\nu + \frac{1}{2}} (\pi - \varphi_1)^{\nu-\frac{1}{2}}  Q
	\simeq D^{\nu+\frac{1}{2}} (\pi - \varphi_1)^{\nu+\frac{1}{2}} \big( \tfrac{(\varphi_0 - \varphi_1) \varphi_0}{(\varphi_0 - \varphi_1) \varphi_0 + D} \big)^{\nu + \frac{1}{2}} \exp \big\{ -\tfrac{\varphi_1^2}{D} \big\},
	\]
	as desired.
\end{proof}

\noindent We have also a result describing the behavior of the function $G_{\frac{\tau}{4}}^{\lambda, \lambda}$ (cf.~\cite[(7)]{NSS21}).

\begin{lemma}
	\label{LNSS3}
	Fix $\lambda \in [0, \infty)$. Then
	\begin{equation*}
	G_{\frac{\tau}{4}}^{\lambda, \lambda}(1, \cos \psi) \simeq \tau^{-\lambda-1} ( \tau + \pi - \psi )^{-\lambda - \frac{1}{2}} \exp\big\{-\tfrac{\psi^2}{\tau} \big\}
	\end{equation*}
	uniformly in $\frac{\tau}{4} \in (0, 1]$ and $\psi \in [0, \pi]$.
\end{lemma}

\noindent Finally, we shall use the following elementary estimate (cf.~\cite[(11)]{NSS21}).

\begin{lemma}
	\label{LNSS4}
	Fix $\kappa \in \RR$. Then 
	\begin{equation*}
	(\tau + \pi - \eta)^{\kappa} \exp \big\{ - \tfrac{\eta^2}{\tau} \big\} \lesssim (\tau + \pi - \theta)^{\kappa} \exp \big\{ - \tfrac{\theta^2}{\tau} \big\}
	\end{equation*}
	uniformly in $\tau \in (0, \infty)$ and $\theta, \eta \in [0,\pi]$ with $\theta \leq \eta$.
\end{lemma}

\section{Proofs of the main results}

We are ready to prove Theorems~\ref{T1}~and~\ref{T2}.
\begin{proof}[Proof of Theorem~\ref{T1}]
	Fix $(x, t), (y, s) \in \mathbb{V}^{d+1}$ and $\tau \in (0,4]$. We consider the formula for $h_{\tau} (W_{\mu, \gamma}; (x, t),  (y, s))$ given in Lemma~\ref{LHK1}.
	By Lemma~\ref{LNSS3} 
	the integrand is comparable to
	\begin{align*}
	\tau^{-2 \alpha -\gamma - \frac{3}{2}}  
	\big( \tau + \pi - \arccos \xi(u, v_1, v_2) \big)^{-2 \alpha - \gamma - 1}
	\exp \big\{ -\tfrac{\arccos^2 \xi(u, v_1, v_2)}{\tau} \big\}, 
	\end{align*}
	where we abbreviate $\xi(x, t, y, s; u, v_1, v_2)$ to $\xi(u, v_1, v_2)$.
	Next, observe that the region of integration $[-1,1]^3$ can be replaced by $[0, 1]^3$ without changing the size of the outcome. Indeed, for given $u, v_1, v_2 \in [0, 1]$, among the points $(\pm u, \pm v_1, \pm v_2)$ it is $(+u, +v_1, +v_2)$ where the smallest value of $\arccos \xi(\pm u, \pm v_1, \pm v_2)$ is attained. Hence, the value of the integrand at this point dominates the remaining ones by Lemma~\ref{LNSS4}.
	Also, we have $\xi(+ u, + v_1, + v_2) \geq 0$ which gives $\tau + \pi - \arccos \xi(+ u, + v_1, + v_2) \simeq 1$. Consequently, 
	\begin{equation*}
	h_{\tau} \big(W_{\mu,\gamma};(x, t),  (y, s) \big) \simeq \tau^{-2 \alpha -\gamma - \frac{3}{2}} \int \limits_{[0, 1]^3}  \exp \big\{ -\tfrac{\arccos^2 \xi(u, v_1, v_2)}{\tau} \big\}
	\, \dint \Pi_{\mu - \frac{1}{2}}(u) \dint \Pi_{\alpha - \frac{1}{2}}(v_1) \dint \Pi_{\gamma}(v_2).
	\end{equation*}
	It remains to estimate the integral above. In order to do so, we will repeatedly use Lemma~\ref{LNSS1} or its variant Lemma~\ref{LNSS2}. 
	Firstly, we use Lemma~\ref{LNSS1} 
	for the variable $v_2$, taking
	\begin{equation*}
	\nu = \gamma,
	\quad
	A = v_1 \sqrt{\tfrac{1}{2} \big( st + \langle x, y \rangle + \sqrt{t^2 - \norm{x}^2} \sqrt{s^2 - \norm{y}^2} u\big)},
	\quad
	B = \sqrt{1 - t} \sqrt{1 - s},
	\quad
	D=\tau.
	\end{equation*}
	Since $\pi - \Phi_{A,B}(1) = \pi - \arccos \xi(u, v_1, 1) \simeq 1$, we obtain that the integral is comparable to 
	\begin{align*}
	\tau^{\gamma + \frac{1}{2}} \int \limits_{[0, 1]^2} \big( \sqrt{1-t}\sqrt{1-s} + \tau \big)^{-\gamma - \frac{1}{2}}
	 \exp \big\{ -\tfrac{\arccos^2 \xi(u, v_1, 1)}{\tau} \big\} \, 
	 \dint \Pi_{\mu - \frac{1}{2}}(u) \dint \Pi_{\alpha - \frac{1}{2}}(v_1).
	\end{align*}
	Now we can extract the factor $\tau^{\gamma + \frac{1}{2}} ( \sqrt{1-t}\sqrt{1-s} \vee \tau )^{-\gamma - \frac{1}{2}}$ and focus on the remaining part of the formula. 
	We apply Lemma~\ref{LNSS1} again, this time for the variable $v_1$, taking  
	\begin{equation*}
	\nu = \alpha - \tfrac{1}{2},
	\quad
	A = \sqrt{1-t} \sqrt{1-s},
	\quad
	B = \sqrt{\tfrac{1}{2} \big( st + \langle x, y \rangle + \sqrt{t^2 - \norm{x}^2} \sqrt{s^2 - \norm{y}^2} u\big)},
	\quad
	D = \tau.
	\end{equation*}
	As before, $\pi - \Phi_{A,B}(1) = \pi - \arccos \xi(u, 1, 1) \simeq 1$ so the resulting expression is
	\begin{align*}
	\tau^{\alpha} \int \limits_{[0, 1]} \Big( \sqrt{\tfrac{1}{2} \big( st + \langle x, y \rangle + \sqrt{t^2 - \norm{x}^2} \sqrt{s^2 - \norm{y}^2} u \big)} + \tau \Big)^{-\alpha}
	\exp \big\{ -\tfrac{\arccos^2 \xi(u, 1, 1)}{\tau} \big\} \, \dint \Pi_{\mu - \frac{1}{2}}(u). 
	\end{align*}
	We now claim that $st + \langle x, y \rangle \ge \sqrt{t^2 - \norm{x}^2} \sqrt{s^2 - \norm{y}^2}$. Indeed, by the Cauchy--Schwarz inequality we have $st + \langle x, y \rangle \ge st - \norm{x} \norm{y} \ge 0$, while the remaining numerical inequality $(st - \norm{x} \norm{y})^2 \ge (t^2 - \norm{x}^2)(s^2 - \norm{y}^2)$ is easy to check. Thanks to that   
	\begin{equation*}
	\sqrt{\tfrac{1}{2} \big( st + \langle x, y \rangle + \sqrt{t^2 - \norm{x}^2} \sqrt{s^2 - \norm{y}^2} u \big)} \simeq \sqrt{st + \langle x, y \rangle}
	\end{equation*}
	and we are left with the expression	
	\begin{equation*}
	\tau^\alpha ( \sqrt{st + \langle x, y \rangle} \vee \tau)^{-\alpha} \int \limits_{[0, 1]} \exp \big \{ -\tfrac{\arccos^2 \xi(u, 1, 1) }{\tau} \big \} \, \dint \Pi_{\mu - \frac{1}{2}}(u).
	\end{equation*}
	Again, it suffices to focus on the integral above.	
	This time we use Lemma~\ref{LNSS2}. Set  
	\[ 
	A = \tfrac{1}{2} \big(st + \langle x, y \rangle \big), \quad
	B = \tfrac{1}{2} \sqrt{t^2 - \norm{x}^2} \sqrt{s^2 - \norm{y}^2}, \quad
	D = \tau, \quad
	E = \sqrt{1-t} \sqrt{1-s},
	\]
	and note that $0 \le B \le A, \, 0 \le E, \, \sqrt{A + B} + E \le 1$. Assume that $\mu, B > 0$. Since 
	$
	1 - u^2 = (1-u)(1+u) \simeq 1-u
	$
	for $u \in [0, 1]$, our task comes down to estimating the integral	
	\begin{equation*}
	\int \limits_{[0, 1]} \exp \big\{ -\tfrac{\arccos^2 (\sqrt{A + Bu} + E )}{D} \big\} \, (1-u)^{\mu-1}  \dint u.
	\end{equation*}
	Observe that substituting $\sqrt{A + Bu} + E = \cos \psi$ we obtain 
	$
	u = \tfrac{1}{B} ( (\cos \psi - E)^2 - A )
	$ 
	and
	$\dint u = \tfrac{2}{B} (\cos \psi - E) (- \sin \psi) \, \dint \psi$.
	Consequently, the integral can be written as
	\begin{align*}
	\int \limits_{\varphi_0}^{\varphi_1} \exp \big\{ - \tfrac{\psi^2}{D} \big\} \Big( 1 - \tfrac{1}{B} \big( (\cos \psi - E)^2 - A \big) \Big)^{\mu - 1} \tfrac{2}{B} (\cos \psi - E) (- \sin \psi) \, \dint \psi,
	\end{align*}
	where $\varphi_0, \varphi_1 \in [0, \frac{\pi}{2}]$ are such that 
	$
	\cos \varphi_0 = \sqrt{A} + E
	$ 
	and 
	$\cos \varphi_1 = \sqrt{A + B} + E$.
	Note that 
	\[
	1 - \tfrac{1}{B} \big( (\cos \psi - E)^2 - A \big) = \tfrac{1}{B} \big(A + B - (\cos \psi - E)^2 \big) = \tfrac{1}{B} \big( (\cos \varphi_1 - E)^2 - (\cos \psi - E)^2 \big)
	\]
	and one can rewrite the last expression as
	\begin{equation*}
	 \tfrac{1}{B} \big( (\cos \varphi_1 - E) + (\cos \psi - E) \big) ( \cos \varphi_1 - \cos \psi ).
	\end{equation*}
	Since  
	$(\cos \varphi_1 - E) + (\cos \psi - E) \simeq \cos \varphi_1 - E = \sqrt{A}$, 
	the considerations above lead to
	\begin{align*}
	\big( \tfrac{\sqrt{A}}{B}  \big)^{\mu} \int \limits_{\varphi_1}^{\varphi_0} \exp \big\{ - \tfrac{\psi^2}{D} \big\} ( \cos \varphi_1 - \cos\psi )^{\mu - 1}  \sin \psi  \, \dint \psi.
	\end{align*}
	Taking $\nu = \mu - \frac{1}{2}$ in Lemma~\ref{LNSS2}, we see that this quantity is comparable to
	\begin{align*}
	\big( \tfrac{\sqrt{A}}{B} \big)^{\mu} D^{\mu} (\pi - \varphi_1)^{\mu} \big( \tfrac{(\varphi_0 - \varphi_1) \varphi_0}{(\varphi_0 - \varphi_1) \varphi_0 + D} \big)^{\mu} \exp \big\{ -\tfrac{\varphi_1^2}{D} \big\}.
	\end{align*}
	Since $\pi - \varphi_1 \simeq 1$ and 
	$(\varphi_0 - \varphi_1) \varphi_0 \simeq \cos\varphi_1 - \cos \varphi_0 = \sqrt{A + B} - \sqrt{A} \simeq \frac{B}{ \sqrt{A}}$, this turns into
	\begin{align*}
 	D^{\mu} \big( \tfrac{B}{\sqrt{A}} + D \big)^{-\mu} \exp \big\{ -\tfrac{\arccos^2 (\sqrt{A + B} + E)}{D} \big\}.
	\end{align*}
	One can easily check that if $\mu = 0$ or $B = 0$, then we end up with the same expression. Finally, combining all the previous estimates, we conclude that the heat kernel is of the size
	\begin{align*}
	\tau^{-\alpha- 1}
	\big( \sqrt{1-t}\sqrt{1-s} \vee \tau \big) ^{-\gamma - \frac{1}{2}}
	\big( \sqrt{st + \langle x, y \rangle} \vee \tau \big)^{-\alpha}
	D^{\mu} \big( \tfrac{B}{\sqrt{A}} + D \big)^{-\mu} \exp \big\{ -\tfrac{\arccos^2 (\sqrt{A + B} + E)}{D} \big\}
	\end{align*}
	and expanding $A, B, D, E$ completes the proof.
\end{proof}

\begin{proof}[Proof of Theorem~\ref{T2}]
	Fix $(x, t), (y, s) \in \mathbb{V}^{d+1}_0$ and $\tau \in (0,4]$.
	We consider the formula for $h_{\tau} (\varphi_{\gamma}; (x, t), (y, s))$ given in Lemma~\ref{LHK2}. 
	By Lemma~\ref{LNSS3} 
	the integrand is comparable to
	\begin{align*}
	\tau^{-\gamma - d + \frac{1}{2}}  
	\big( \tau + \pi - \arccos \xi(v_1, v_2) \big)^{- \gamma - d + 1}
	\exp \big\{ -\tfrac{\arccos^2 \xi(v_1, v_2)}{\tau} \big\}, 
	\end{align*}
	where we abbreviate $\xi(x, t, y, s; v_1, v_2)$ to $\xi(v_1, v_2)$.
	As before, thanks to Lemma~\ref{LNSS4} the region of integration $[-1,1]^2$ can be replaced by $[0, 1]^2$.
	Also, if $v_1, v_2 \in [0,1]$, then $\xi(v_1, v_2) \ge 0$ which gives $\tau + \pi - \arccos \xi(v_1, v_2) \simeq 1$. Consequently, 
	\begin{equation*}
	h_{\tau} \big(\varphi_{\gamma};(x, t),  (y, s) \big) \simeq \tau^{-\gamma - d + \frac{1}{2}} \int \limits_{[0, 1]^2}  \exp \big\{ -\tfrac{\arccos^2 \xi(v_1, v_2)}{\tau} \big\}
	\, \dint \Pi_{\frac{d-3}{2}}(v_1) \dint \Pi_{\gamma}(v_2).
	\end{equation*}
	To estimate the integral above, we use Lemma~\ref{LNSS1} for the variable $v_2$, taking
	\begin{equation*}
	\nu = \gamma,
	\quad
	A = v_1 \sqrt{ \tfrac{1}{2} ( st + \langle x, y \rangle ) },
	\quad
	B = \sqrt{1-t} \sqrt{1-s},
	\quad
	D = \tau.
	\end{equation*}
	Since $\pi - \Phi_{A,B}(1) = \pi - \arccos \xi(v_1, 1) \simeq 1$, we obtain that the integral is comparable to 
	\begin{align*}
	\tau^{\gamma + \frac{1}{2}} \int \limits_{[0, 1]} \big( \sqrt{1-t}\sqrt{1-s} + \tau \big)^{-\gamma - \frac{1}{2}}
	\exp \big\{ -\tfrac{\arccos^2 \xi(v_1, 1)}{\tau} \big\} 
	\, \dint \Pi_{\frac{d-3}{2}}(v_1).
	\end{align*}
	We can extract the factor $\tau^{\gamma + \frac{1}{2}} ( \sqrt{1-t}\sqrt{1-s} \vee \tau )^{-\gamma - \frac{1}{2}}$ and focus on the remaining part of the formula. 
	We apply Lemma~\ref{LNSS1} again, this time for the variable $v_1$, taking  
	\begin{equation*}
	\nu = \tfrac{d-3}{2},
	\quad
	A = \sqrt{1-t} \sqrt{1-s},
	\quad
	B = \sqrt{\tfrac{1}{2} \big( st + \langle x, y \rangle \big)},
	\quad
	D = \tau.
	\end{equation*}
	As before, $\pi - \Phi_{A,B}(1) = \pi - \arccos \xi(1, 1) \simeq 1$ so the resulting expression is
	\begin{align*}
	\tau^{\frac{d}{2}-1} \Big( \sqrt{\tfrac{1}{2} \big( st + \langle x, y \rangle \big)} + \tau \Big)^{-\frac{d}{2}+1}
	\exp \big\{ -\tfrac{\arccos^2 \xi(1, 1)}{\tau} \big\}. 
	\end{align*} 
	Combining all the previous estimates completes the proof.
\end{proof}

\end{document}